\documentclass[intlimits,a4paper]{amsart}
\usepackage[utf8]{inputenc} 
\usepackage{epigraph}
\usepackage[all]{xy}
\usepackage{amssymb}

\newcommand{\Q}{{\bf Q}}

\newcommand{\Z}{\mathbb{Z}}
\newcommand{\C}{\mathbb{C}}
\newcommand{\F}{\mathbb{F}}
\newcommand{\N}{\mathbb{N}}

\newcommand{\R}{\mathbb{R}}

\newcommand{\cE}{\mathcal{E}}

\newtheorem{theorem}{Theorem}
\newtheorem{lemma}[theorem]{Lemma}

\newtheorem*{rem}{Remark}
\newenvironment{remark}{\begin{rem}\begin{rm}}{\end{rm}\end{rem}}

\newtheorem{corollary}[theorem]{Corollary}

\begin{document}
\title
{Weakly holomorphic modular forms for some moonshine groups}  
\author[M. Lahr, R. Schulze-Pillot]{Martina Lahr, Rainer Schulze-Pillot} 

\maketitle
\begin{abstract}
In an article in the Pure and Applied Mathematics Quarterly in 2008, Duke
and Jenkins investigated a certain natural basis of the space of
weakly holomorphic modular forms for the full modular group
$SL_2({\bf Z})$. We show here that their results can be generalized to
certain moonshine groups, also allowing characters that are real on the
underlying subgroup $\Gamma_0(N)$.  
\end{abstract}
 \section{Introduction}
The name ``moonshine groups'' is commonly used for the discrete subgroups of ${\rm PSL}_2(\R)$ which
appear in connection with the famous moonshine conjectures of Conway and Norton \cite{conway-norton}, see
\cite{con-mckay-sebbar}. In connection with the study of euclidean lattices with high minima, it has been
noticed by Quebbemann in \cite{que1,que2} that for some of these, namely the groups $\Gamma_0(N)^+$ generated
by the Hecke group $\Gamma_0(N)$ and the Atkin-Lehner involutions normalizing it for $N=2,3,4,6,7,11,14,15,23$,
the theory of modular forms is very similar to that for the full modular group. In particular, for these
groups one can, using Dedekind's eta function, construct an analogue $\Delta_N$ of the discriminant function
$\Delta$, i.e., a cusp form with a simple zero at infinity and no zeros in the upper half plane. This made it
possible to transfer the notion of an extremal modular form to these groups and to study lattices which have
these extremal modular forms as their theta series.
 \vspace{0.2cm}\\
The extremal modular form for the full modular group is one of the basis functions $f_{k,m}$ of the space of weakly holomorphic modular forms of weight $k$ for the full modular group whose properties have been investigated by Duke and Jenkins in \cite{duke-jenkins}.
In this note we want to show that some of these results from \cite{duke-jenkins} can also
be generalized to the special  moonshine groups mentioned above; we can even admit quadratic characters on $\Gamma_0(N)$, suitably extended to the larger group. We plan to use our results  for the study of extremal lattices of these levels in future
work.  A simple, but so far apparently unnoticed, fact which makes this transfer possible is stated in 
Corollary \ref{hauptmodule_cor}: If $j_N$ is the hauptmodule for $\Gamma_0(N)^+$, then $q \cdot \frac{dj_N}{dq} \Delta_N$
is an Eisenstein series (as is well-known for level $N=1$). This was
initially established by the first named author for prime levels $N$ as
above with the help of
computer calculations and the available explicit $q$ - expansions  for
$j_N$ in her 2013 master's thesis Universität des Saarlandes. This
fact can of course conversely be used in order to compute the $q$ -
expansion of $j_N$ from the known expansions of $\Delta_N$ and the
relevant Eisenstein series in the cases under investigation here.
 \section{Preliminaries}
For a square free positive integer $N$ we consider the Hecke subgroup $\Gamma_0(N)=\{{a\,b\choose c\,d} \in {\rm SL}_2(\Z)~|~
c\equiv 0(N)\}$ of ${\rm SL}_2(\Z)$ and for $m|N$ the Atkin-Lehner involutions 
$W_m =\begin{pmatrix} mx & y \\ Nz & mw \end{pmatrix}$ with $\det W_m = m$ normalizing $\Gamma_0(N)$. By $\Gamma_0(N)^+ \subseteq
{GL}_2(\Q)$ we denote the subgroup generated by $\Gamma_0(N)$ and the $W_m$ for $m|N$. It is well-known that the
quotient $\Gamma_0(N)^+ \setminus  H$ of the upper half plane $H \subseteq \C$ by $\Gamma_0(N)^+$ has only one cusp class. 
Moreover, if we restrict $N$ to be one of the integers $2,3,5,6,7,11,14,15,23,$ the compactified quotient has genus zero and
$\Delta_N(z) = \prod_{m\mid N} \eta(mz)^{24/\sigma_1(N)}$, where $\eta$ is the Dedekind eta function and $\sigma_1(N)$ the divisor
sum, is a cusp form of weight $k_1(N) = \frac{12\sigma_0(N)}{\sigma_1(N)}$ for $\Gamma_0(N)$ with character
$\left(\frac{(-N)^{k_1(N)}}{a}\right)$, having a zero of order 1 in the cusp and having no zeros in $H$, see
\cite{que1,que2,harada}.
 \vspace{0.3cm}\\
Indeed, one checks that $\Delta_N$ is modular for the extended group
$\Gamma_0(N)^+$ with character $\chi^{(N)} = \psi^{k_1(N)}$,
where $\psi{a\,b\choose c\,d} = (\frac{-N}{a})$ for ${a\,b\choose
  c\,d} \in \Gamma_0(N)$ if $N$ is a prime congruent to $3$ modulo $4$
and $\psi\vert_{\Gamma_0(N)}=1$ otherwise,
$\psi(W_N) = i^{-1}$, $\psi(W_2) = 1$ for $N \in \{6,14\}$, $\psi(W_5) = 1$ for $N = 15$. (Notice that for the composite levels treated here $k_1(N)$ is even and $\psi^{k_1(N)}$ trivial on $\Gamma_0(N)$, so that there is no ambiguity in defining the involutions.)

\section{Modular forms for $\Gamma_0(N)^+$}
For an integer $k$ and a real character $\chi$ on $\Gamma_0(N)$ with $\chi(-1) = (-1)^k$ we consider the Eisenstein series
$f_k(z; \chi_1,\chi_2)$ defined in 4.7.6 of \cite{miyake} and the space $ {\cE}_k(N,\chi)$ given in 4.7.17 of 
\cite{miyake} generated
by the $f_k(\ell z; \chi_1,\chi_2)$ for characters $\chi_1 \bmod N_1$, $\chi_2 \bmod N_2$ with $\chi = \chi_1\chi_2$, $\ell N_1N_2|N$ 
and $\chi_i$ primitive modulo $N_i$ or $k = 2$, $\chi_1$ and $\chi_2$ trivial, $N_1 = 1$ and $N_2$ a prime.
 \vspace{0.3cm}\\
By Theorem 4.7.2 of \cite{miyake} the space ${\cE}_k(N,\chi)$ is the orthogonal complement  with respect to the Petersson inner product of the space of cusp forms of weight  $k$ for
$\Gamma_0(N)$ with character $\chi$ .
 \vspace{0.3cm}\\
Moreover, the proof of Theorem 4.7.2 of \cite{miyake} shows that for the special $N$ considered here the dimension 
$d_k(N,\chi)$ of ${\cE}_k(N,\chi)$ is given by
 $$\begin{array}{lll}
 d_0(N,\chi) &=& \left\{ \begin{array}{cc}
 1 & \chi = {\bf 1 }\\
 0 & \chi \not= {\bf 1}
 \end{array}\right.,
 \vspace{0.2cm}\\
 d_1(N,\chi) &=& \left\{ \begin{array}{cl}
 1 & \mbox{if } \chi \not= {\bf 1} \\
 0 & \mbox{otherwise}
 \end{array}\right.
 \end{array}$$
if $N$ is prime, $N \equiv 3(4)$,
 $$d_1(N,\chi) = \left\{ \begin{array}{cl}
 2 & \mbox{if } \chi\not= {\bf 1}\\
 0 & \mbox{otherwise}
 \end{array}\right.$$
if $N = p_1p_2$ is a product of two distinct primes, 
 $$d_2(N,\chi) = \left\{ \begin{array}{cl}
 1 & \mbox{if } \chi = {\bf 1}\\
 2 & \mbox{otherwise}
 \end{array}\right.$$
if $N$ is prime,
 $$d_2(N,\chi) = \left\{ \begin{array}{cl} 
 3 & \mbox{if } \chi = {\bf 1}\\
 4 & \mbox{otherwise}
 \end{array}\right.$$
if $N = p_1p_2$,
 $$d_k(N,\chi) = \left\{\begin{array}{cl}
 2 & \mbox{if $N$ is prime}\\
 4 & \mbox{if $N = p_1p_2$}
 \end{array}\right.$$
for $k \geq 3$.
 \vspace{0.3cm}\\
In order to study the action of the involutions on the above spaces of Eisenstein series we restrict attention  for composite levels to characters $\chi$ on $\Gamma_0(N)$ for which the operators $W_m$ of \cite{atkin-li} commute, so that it makes sense to split the space of modular forms for $\Gamma_0(N)$ with character $\chi$ into eigenspaces with respect to the involutions.
For our restricted list of $N$ this means that either $\chi$ is trivial or $N=14$ and $\chi$ the Legendre symbol modulo $7$ or $N=15$ and $\chi$ the Legendre symbol modulo $15$.

Combining the above results with Weisinger's \cite{weisinger} calculation of the action of the $W_m$ on Eisenstein series, we obtain then
 \begin{lemma}\label{eisensteinlemma}
Let $k \in \N$, $N$ a prime or a product of two distinct primes, let $\chi$ as above be a character on $\Gamma_0(N)^+$ which is real on
$\Gamma_0(N)$ with $\chi(-1) = (-1)^k$ and such that (for composite $N$) the involutions $W_m$  commute.
Then ${\cE}_k(N,\chi)^+ := M_k(\Gamma_0(N)^+,\chi) \cap {\cE}_k(N,\chi|_{\Gamma_0(N)})$
has dimension $1$ except for $k=1$, $\chi \not= \psi$, $k = 2$, $\chi = {\bf 1}$ (in which cases 
${\cE}_k(N,\chi)^+ = \{0\})$.\\
In the one-dimensional cases we write $E_k^{(\chi)}$ for the element of ${\cE}_k(N,\chi)^+$ whose Fourier expansion at $\infty$ has
constant term $1$ (if $N$ is fixed).      
 \end{lemma}
Similar to the case of modular forms for the full modular group, multiplication by $\Delta_N$ defines an injective linear map from $M_k(\Gamma_0(N)^+, \chi)$ to $M_{k+k_1(N)}(\Gamma_0(N)^+, \chi\psi^{k_1(N)})$ whose image is the space of cusp forms. It will therefore be useful to twist a given character $\chi$ by $\psi^r$ if one wants to connect modular forms of weight $k$ with those of weight $k+r$, in particular if $r=\ell k_1(N)$ is a multiple of $k_1(N)$. 
 
 \begin{theorem}\label{k_min_lemma}
Let $k \in \N$, $N$ from the list in the previous section, $\chi$ as above. Denote by $k' = k'(N,\chi,k)$ the smallest
non negative integer $k' \equiv  k \bmod k_1(N)$ with $M_{k'}(\Gamma_0(N)^+,\chi \psi^{k'-k}) \not= \{0\}$. Then one has
$$k'= \left\{\begin{array}{ll}
 k_1(N) & \mbox{if } k \equiv 0 \bmod k_1(N),\: \chi \not= \psi^k \\
 1+k_1(N) & \mbox{if } k \equiv 1 \bmod k_1(N),\\ 
 & \chi \psi^{-k}|_{\Gamma_0(N)\cup W_N\Gamma_0(N)} \not= {\bf 1} \\
 2+k_1(N) & \mbox{if } k \equiv 2 \bmod k_1(N),\\
 & \chi \psi^{2-k} = {\bf 1}\\ \tilde{k} \in \{0,\ldots,k_1(N)-1\} & \mbox{in all other cases}
 \end{array}\right.$$
In all cases
 $ M_{k'} (\Gamma_0(N)^+,\chi \psi^{k'-k}) = {\cE_{k'}}
 (\Gamma_0(N)^+,\chi \psi^{k'-k})={\cE_{k'}}(N,\chi)^+$ is one-dimensional.
 \end{theorem}
 \begin{proof}
Since the existence of a cusp form $f$ in $M_{k'}(\Gamma_0(N)^+,\chi \psi^{k'-k})$ can be ruled out by considering
$\frac{f}{\Delta_N^r}$ (where $r$ is the order of the zero of $f$ at
$\infty$) and the minimality of $k'$, this follows from the previous lemma.
 \end{proof}
 \begin{corollary}
Let $k,\chi,N$ be as above, $k' = k'(N,\chi,k)$, assume $k \geq k'$. Then $\dim M_k(\Gamma_0(N)^+,\chi) = 1+
\frac{k-k'}{k_1(N)}$ and the forms 
 $$E_{k'}^{(\chi \psi^{k'-k})}(E_{k_1(N)}^{(\psi^{k_1(N)})})^r \Delta_N^s \mbox{ with } r+s = \frac{k-k'}{k_1(N)}$$
form a basis of this space.
 \end{corollary}
 \begin{proof}
By induction on $\frac{k-k'}{k_1(N)}$, the case $k=k'$ being given by
the previous Lemma.
 \end{proof}
 \begin{remark}
If $N$ is not restricted to our list of special levels, the dimension of the space of cusp forms for $\Gamma_0(N)^+$ with
character $\chi$ and weight $k$ can be calculated using the trace
formula from \cite{skoruppa-zagier}. If $N=1$ and $\chi$ is trivial,
the corollary gives
the familiar dimension formula for modular forms for the full modular group ${\rm SL}_2(\Z)$.
 \end{remark}
 \begin{corollary}\label{hauptmodule_cor}
Let $k$, $N$ be as before, denote by $j_N$ the hauptmodule for the modular curve $\overline{\Gamma_0(N)^+\setminus H}$ of genus
zero, normalized to $j_N = q^{-1}+c_1q+ \cdots$ (with $q = \exp(2\pi
i\tau), \tau \in H)$. Then $-\Delta_N\cdot q  \frac{dj_N}{dq}
= E_{2+k_1(N)}^{(\chi^{(N)})}$.
 \end{corollary}
 \begin{proof}
$-\Delta_N \cdot q \cdot \frac{dj_N}{dq}$ is a modular form of weight $2+k_1(N)$ and character $\chi^{(N)}=\psi^{k_1(N)}$ for
$\Gamma_0(N)^+$ with constant term $1$ at $\infty$. By the theorem, $M_{2+k_1(N)}(\Gamma_0(N)^+,\chi^{(N)})$ is one-dimensional,
spanned by $E_{2+k_1(N)}^{(\chi^{(N)})}$.
 \end{proof}
 \begin{corollary}\label{eisenstein_zeros}
 \begin{itemize}
 \item[a)] Let $N=2$. Then all zeros of the Eisenstein series
   $E_k^{(\chi)}$ of weight $k\le 4$ in the fundamental domain
$\F^{\ast}(2)$ of \cite{mns} are on the arc
 $$\overline{A_2^{\ast}} = \{z \in \C~|~|z| =
 \frac{1}{\sqrt{2}},\frac{\pi}{2} \leq \mbox{\rm arg}(z) \leq \frac{3\pi}{4}\}.$$
 \item[b)] Let $N = 3$. Then all zeros of the Eisenstein series 
of the Eisenstein series
   $E_k^{(\chi)}$ of weight $k\le 4$ in the fundamental domain
   $\F^{\ast}(3)$ of \cite{mns} are on the arc 
 \begin{equation*}
 \overline{A}_3^{\ast} = \{z \in \C~|~ |z| = \frac{1}{\sqrt{3}},\frac{\pi}{2} \leq {\rm arg}(z) \leq \frac{5\pi}{6}\}.
 \end{equation*}
 \item[c)] Let $N = 5$. Then all zeros of the Eisenstein series
   $E_k^{(\chi)}$ of weight $k\le 4$ for $\Gamma_0(5)^+$ in the fundamental
domain 
 \begin{equation*}
 \begin{array}{lll}
F^{\ast}(5) &=& \{z \in \C~|~|z| \geq \frac{1}{\sqrt{5}}, |z+\frac{1}{2}| \geq \frac{1}{2\sqrt{5}},-\frac{1}{2} \leq {\rm Re}(z) 
   \leq 0\}\\
 &\cup& \{z \in \C~|~|z| > \frac{1}{\sqrt{5}}, |z+\frac{1}{2}| > \frac{1}{2 \sqrt{5}}, 0 \leq {\rm Re}(z) < \frac{1}{2}\}
 \end{array}
 \end{equation*}
from \cite{shig} (notice that this is misprinted in \cite{shig}) are on the arc $\overline{A}_5^{\ast} =
\F^{\ast}(5) \cap \{z \in \C~|~|z| = \frac{1}{\sqrt{5}}$ or $|z+\frac{1}{2}| = \frac{1}{2\sqrt{5}}$.
 \end{itemize}
 \end{corollary} 
 \begin{proof} The assertions for $\chi={\bf1}$ have been proven in
   \cite{mns, shig}, so we have to reduce our assertions for
   nontrivial $\chi$ to the cases treated there.
 \begin{itemize}
 \item[a)] Since $E_2^{(\psi^2)} E_4^{(\psi^2)} = E_6^{({\bf 1})}$ by the one-dimensionality of  
$M_6(\Gamma_0(2)^+,{\bf 1})$, this follows from the result of \cite{mns} for the zeros of $E_6^{({\bf 1})}$.
 \item[b)] Again, by dimension reasons we have 
 \begin{equation*}
 \begin{array}{l}
 (E_1^{(\psi)})^4 = E_4^{({\bf 1})} = (E_2^{(\psi^2)})^2,\\
 E_3^{(\psi)} E_3^{(\psi^3)} = E_6^{({\bf 1})} = E_2^{(\psi^2)} E_4^{(\psi^2)},
 \end{array}
 \end{equation*}
and the assertion follows from the results of \cite{mns} for $E_4^{({\bf 1})}, E_6^{({\bf 1})}$.
 \item[c)] Again, by dimension reasons we can multiply any of the Eisenstein series of weight 2 with a 
suitable Eisenstein series of weight 4 to obtain $E_6^{({\bf 1})}$, and the assertion follows from the result
of \cite{shig} [Prop. 2.3] for $E_6^{({\bf 1})}$ (for a correction see  {\tt  arXiv:math/0607409v3}).
 \end{itemize}
 \end{proof}
 \begin{remark}
At present we see no way to obtain similar results for non trivial
$\chi$ and higher weights. Also, it seems to be difficult to treat the
other levels in our list.   
 \end{remark}
 \section{Weakly holomorphic modular forms}
As in \cite{duke-jenkins}, we call a holomorphic function on $\mathbb H$ which is meromorphic in the cusp of 
$\Gamma_0(N)^+\setminus \mathbb H$ and transforms like a modular form of weight $k$ and character $\chi$ under
$\Gamma_0(N)^+$ a weakly holomorphic modular form for $\Gamma_0(N)^+$, and denote by $\mathcal M_k(\Gamma_0(N)^+,\chi)$ 
the space of such weakly holomorphic modular forms.
 \begin{lemma}
Let $N,k,\chi \in \Z$ be as in the previous section and write $k = k'+\ell \cdot k_1(N)$ with
$k' = k'(N,\chi,k)$ as in Theorem \ref{k_min_lemma}.  Then for each integer $m \geq -\ell$, there exists a unique 
$f_{k,m}^{(\chi)} \in \mathcal M_k(\Gamma_0(N)^+,\chi)$ with $q$-expansion of the form
 \begin{equation*}
 f_{k,m}^{(\chi)} (\tau) = q^{-m}+O(q^{\ell+1}),
 \end{equation*}
and $f_{k,m}^{(\chi)}$ can be written as 
 \begin{equation*}
f_{k,m}^{(\chi)} = \Delta_N^{\ell} E_{k'}^{(\chi \psi^{k'-k})} \cdot F_{k,\chi,D,N}(j_N)
 \end{equation*}
where $F_{k,\chi,D,N}$ is a monic polynomial of degree $D = \ell+m$ with integer coefficients.\\
The $f_{k,m}^{(\chi)}$ with $m \geq -\ell$ form a basis of $\mathcal
M_k(\Gamma_0(N)^+,\chi)$. The coefficients $a_k^{(\chi)} (m,n)q^n$ in 
$f_{k,m}^{(\chi)}(\tau) = q^{-m}+\sum_{n=\ell+1}^{\infty} a_k^{(\chi)}
(m,n)q^n$ are integers.\\
Following \cite{duke-jenkins}, we write $f_k^{(\chi)}: = f_{k,-\ell}^{(\chi)} = E_{k'}^{(\chi
  \psi^{k'-k})}\,\Delta_N^{\ell}$ in the sequel.
 \end{lemma}
 \begin{proof}
This follows from the results in the previous section in the same way as in \cite{duke-jenkins} in the case of level 1.
 \end{proof}
 \begin{lemma}\label{hauptmodule_lemma_2}
Let $k,N,\chi$ be as before with  $k=k'(N,\chi,k)$, denote by $\tilde{\chi}$ the character on $\Gamma_0(N)^+$ 
with $\chi \tilde{\chi} = \chi^{(N)}=\psi^{k_1(N)}$. Then one has
 \begin{equation*}
E_k^{(\chi)} E_{2+k_1(N)-k}^{(\tilde{\chi})} = E_{2+k_1(N)}^{(\chi^{(N)})} = -q \cdot \frac{dj_N}{dq} \Delta_N.
 \end{equation*}
 \end{lemma}
 \begin{proof}
Checking the cases in Theorem \ref{k_min_lemma} we see that 
$k'(N,\tilde{\chi},2+k_1(N)-k)) = 2+k_1(N)-k$, so that $M_{2+k_1(N)-k}
\not= \{0\}$, i.e., $E_{2+k_1(N)-k}^{(\tilde{\chi})}$ is well defined.
 \vspace{0.2cm}\\
Since $M_{2+k_1(N)} (\Gamma_0(N)^+,\chi^{(N)})$ is one-dimensional, generated by $E_{2+k_1(N)}^{(\chi^{(N)})}$,
the assertion follows (use Corollary \ref{hauptmodule_cor} for the second equality).
 \end{proof}
 \begin{lemma}
With notation as before (in particular $q = \exp(2\pi i \tau)$) we have
\begin{eqnarray*}
 f_{k,m}^{(\chi)}(z) &=& \frac{1}{2\pi i} \int_{C} \frac{\Delta_N^{\ell}(z) E_{k'}^{(\chi \psi^{-\ell k_1(N)})}(z) 
    E_{2+k_1(N)-k'}^{(\chi^{-1}\psi^{(1+\ell)k_1(N)})}(\tau) q^{-m-1} dq}{\Delta_N^{1+\ell}(\tau) (j(\tau)-j(z))}\\
 &=& \frac{1}{2\pi i} \int_{C} \frac{f_k^{(\chi)}(z)
   f_{2-k}^{(\chi^{-1})}(\tau)}{j(\tau)-j(z)} q^{-m-1} dq\\
&=&\frac{1}{2\pi i}
 \int_{C} \frac{f_k^{(\chi)}(z)
   f_{2-k}^{(\chi\psi^{-2k})}(\tau)}{j(\tau)-j(z)} q^{-m-1} dq, 
 \end{eqnarray*}
where $C$ is a suitable (i.e. sufficiently small) circle around zero in the $q$-plane, oriented 
counterclockwise.
 \end{lemma}
 \begin{proof}
The first equality is proved in the same way as Lemma 2 in
\cite{duke-jenkins}:
We
write \begin{equation*}F_{k,\chi,D,N}(j_N(\tau))=\frac{f_{k,m}^{(\chi)}(\tau)}{\Delta_N^{\ell}(\tau)
    E_{k'}^{(\chi \psi^{k'-k})}(\tau)},\end{equation*}
use $ f_{k,m}^{(\chi)} (\tau) = q^{-m}+O(q^{\ell+1})$ and the Cauchy
integral formula to obtain an expression 
\begin{equation*}
  F_{k,\chi,D,N}(\zeta)=\int_{C'}\frac{q^{-m}}{\Delta_N^{\ell}(j_N)E_{k'}^{(\chi \psi^{k'-k})}(j_N)(j_N-\zeta)}dj_N
\end{equation*} of $ F_{k,\chi,D,N}(\zeta)$ as an integral around a
suitable counterclockwise circle around $0$ in the $j_N$-plane, and
replace the equality  
$-q \frac{dj}{dq} \Delta = E_{14}=E_{k'}E_{14-k'}$ used in
\cite{duke-jenkins} with the equality 
 \begin{equation*}
E_{k'}^{(\chi\psi^{-\ell k_1(N)})}(\tau)
E_{2+k_1(N)-k'}^{(\widetilde{\chi\psi^{-\ell k_1(N)}})}(\tau) = E_{2+k_1(N)}^{(\chi^{(N)})}(\tau) = -q \cdot \frac{dj_N}{dq} \Delta_N(\tau)
 \end{equation*}
(where $\widetilde{\chi\psi^{-\ell k_1(N)}}=\chi^{-1}\psi^{(1+\ell)k_1(N)}$)
from Lemma \ref{hauptmodule_lemma_2} to perform a change of variables
from the variable $j_N$ to the variable $q$.
 \vspace{0.3cm}\\
The second equality then follows (again as in \cite{duke-jenkins}) from the equalities
 \begin{eqnarray*}
 f_k^{(\chi)}(z) &=& \Delta_N^{\ell}(z) E_{k'}^{(\chi \psi^{-\ell k_1(N)})}(z),\\
 f_{2-k}^{(\chi^{-1})} (\tau) &=& \Delta_N^{-1-\ell}(\tau) 
   E_{2+k_1(N)-k'}^{(\chi^{-1}\psi^{(1+\ell)k_1(N)})}(\tau),
 \end{eqnarray*}
where we have written $\widetilde{\chi \psi^{-\ell K_1(N)}}$ as
$\chi^{-1}\psi^{(1+\ell)k_1(N)}$ and used that one has
$2-k = 2+k_1-k'(N,\chi,k)-(1+\ell)k_1 =
k'(N,\chi,2-k)-(1+\ell)k_1$. Notice also that one has
$\chi^{-1}=\chi\psi^{-2k}$ in all cases.
 \end{proof}
 \begin{remark}
The lemma above is used in \cite{jenkins-rouse} in order to prove estimates for the weights in which
the extremal modular form for ${\rm SL}_2(\Z)$ has only non-negative coefficients. We plan to
investigate this question for the extremal modular form of level $N$ and character $\chi$ as studied above;
in fact this was the motivation for the study of this situation in the present article.
 \end{remark}
 \begin{theorem}
With notations as before we have
 \begin{equation*}
 \sum_{m \geq -\ell} f_{k,m}^{(\chi)}(z) q^m = \frac{f_k^{(\chi)}(z)f_{2-k}^{(\chi \psi^{-2k})}}
   {j(\tau)-j(z)}.
 \end{equation*}
 \end{theorem}
 \begin{proof}
This is proved as in the proof of Theorem 2 in \cite{duke-jenkins},
using Cauchy's integral formula and the lemma above.
 \end{proof}
 \begin{remark}
Notice that in level $N\not= 1$ the characters $\chi$ belonging to
$f_k$ and $\chi\psi^{-2k}=\chi\psi^{(2-k)-k+2}$ to $f_{2-k}$ are not from the
same series of characters $\chi \psi^r$ ($r \in \Z$) but instead are related by switching from this series
to the ``dual'' series of ${\chi \psi^{r+2}}$ ($r \in \Z$).
 \end{remark}
 \begin{corollary}
With notations as before we have for all integers $m,n$ the equality
 \begin{equation*}
 a_k^{(\chi)} (m,n) = -a_{2-k}^{{(\chi \psi^{-2k})}} (n,m).
 \end{equation*}
 \end{corollary}
 \begin{proof}
Again as in \cite{duke-jenkins}.
 \end{proof}
 \begin{theorem}
With notations as before let $k$ be such that $k = k'(N,\chi,k)$, let $m,n \in \Z$ with 
 \begin{equation*}
{\rm gcd}(m,n) = 1 = {\rm gcd}(N,m).
 \end{equation*}
Then $n^{k-1}~|~a_k^{(\chi)} (m,n)$.
 \end{theorem}
 \begin{proof}
This is proved in the same was as theorem 3 of \cite{duke-jenkins}, using the fact that the condition
$k = k'(N,\chi,k)$ implies that $M_k(\Gamma_0(N)^+,\chi)$ is one-dimensional, spanned by the Eisenstein
series in it.
 \end{proof}
 \begin{remark}
It is not clear whether it is possible to generalize the result of Theorem 1 of \cite{duke-jenkins} on zeros of
the $f_{k,m}$ to the present situation. A start toward such a generalization is given by Corollary \ref{eisenstein_zeros}.
 \end{remark}

\end{document}